\documentclass[a4paper,11pt]{article} 
\usepackage{amsthm}
\usepackage{amsthm, amsmath, amssymb, url,tikz,enumerate}
\usepackage{hyperref}
\usepackage{pgfplots}
\pgfplotsset{compat=newest}
\usepackage{geometry}
\usepackage{mathrsfs}
\usepackage{caption}
\usepackage{graphicx}
\newtheorem{theorem}{Theorem}[section]
\newtheorem{lemma}[theorem]{Lemma}

\newtheorem{corollary}[theorem]{Corollary}

\usepackage[numbers,sort&compress]{natbib}

\title{Sharp upper bounds on the $A_\alpha$-spectral radius of graphs}

\author{Zhen-Mu Hong$^{\,\rm a}$, Zheng-Jiang Xia$^{\,\rm b}$\thanks{Corresponding author. E-mail addresses:  zmhong@mail.ustc.edu.cn (Z.-M. Hong), xzj@mail.ustc.edu.cn (Z.-J. Xia), zhiqiao@sicnu.edu.cn (Z. Qiao).}, Zhi Qiao$^{\,\rm c}$\\
{\footnotesize$^{\rm a}$School of Statistics and Applied Mathematics, Anhui University of Finance \& Economics,} \\{\footnotesize Bengbu 233030, China}\\
{\footnotesize$^{\rm b}$School of Finance, Anhui University of Finance \& Economics, Bengbu 233030, China}\\
{\footnotesize$^{\rm c}$School of Mathematical Sciences, Sichuan Normal University, Chengdu 610068, China}\\
}

\date{}

\begin{document}

\maketitle

\parbox{14.5cm}{
\begin{center}
\textbf{Abstract}
\end{center}

Let $G$ be a simple graph with degree diagonal matrix $D(G)$ and adjacency matrix $A(G)$. The signless Laplacian matrix of $G$ is defined as $Q(G)=D(G)+A(G)$. For a real number $\alpha \in [0, 1]$, Nikiforov (2017) proposed the $A_\alpha$-matrix of a graph $G$ as $A_{\alpha}(G)=\alpha D(G)+(1-\alpha)A(G)$. The $A_\alpha$-spectral radius of $G$, denoted by $\rho_\alpha(G)$, is the largest eigenvalue of $A_\alpha(G)$, where $\rho_0(G)=\rho(G)$ is the spectral radius of $A(G)$ and $2\rho_{\frac{1}{2}}(G)=q(G)$ is the spectral radius of $Q(G)$.  
Sun and Das (2020) proved that for any non-isolated vertex $v$ of degree $d_v$, $\rho^2(G)-\rho^2(G-v) \leq 2 d_v-1$, 
which confirmed the conjecture originally posed by Guo, Wang, and Li (2019). Recently, Liu and Ning (2026) provided a short and self-contained proof of this inequality. In this paper, we establish the corresponding result for $\rho_\alpha(G)$. As a corollary, for every $k\in [0,d_v+1]$, we have 
$$
\rho^2(G)- \rho^2(G-v) \leq 2d_v-1 +(k-2)\left(\frac{d_v}{\rho(G)}-1\right).
$$
This inequality coincides with that of Sun and Das when $k=2$, and is strictly sharper than theirs whenever $k\neq 2$ and $d_v\neq \rho(G)$. 
We also give a short proof of the inequality 
$\rho_{\alpha}(G)-\rho_{\alpha}(G-v)\leq \alpha +\frac{(1-\alpha)^2d_v}{\rho_{\alpha}(G)-\alpha d_v}$, which is obtained by Wang and She (2022). Moreover, we obtain a unified generalization of Hong, Shu and Fang's inequality for $\rho(G)$ and Nikiforov's inequality for $q(G)$ in terms of $\rho_\alpha(G)$.


}

\section{Introduction}

Let $G=(V(G),E(G))$ be a simple graph of order $n$ with vertex-set $V(G)$ and edge-set $E(G)$. For a vertex $v\in V(G)$, let $N(v)$ be the set of neighbors of $v$ in $G$, and  define the closed neighborhood $N[v]=N(v)\cup \{v\}$. The degree $d(v)$ of $v$ is defined as the cardinality of $N(v)$. For any subset $U\subset V(G)$, let $G[U]$ denote the subgraph of $G$ induced by $U$. For any $v\in V(G)$, we use $G-v$ to denote the graph obtained
from $G$ by deleting the vertex $v$ and all edges incident to $v$. 

The adjacency matrix $A(G)=(a_{ij})$ of $G$ is defined to be the $n\times n$ matrix $(a_{ij})$, where $a_{ij}=1$ if $v_iv_j\in E(G)$, and $a_{ij}=0$ otherwise. It is well known that the adjacency matrix is real and symmetric. The adjacency spectral radius of $G$, denoted by $\rho(G)$, is the spectral radius of $A(G)$. The signless Laplacian matrix of $G$ is defined as $Q(G)=D(G)+A(G)$, where $D(G)$ is the degree diagonal matrix of $G$. The signless Laplacian spectral radius $q(G)$ of $G$ is the spectral radius of $Q(G)$. 
In 2017, Nikiforov \cite{niki17} proposed the $A_\alpha$-matrix of a graph $G$ as 
$$
A_{\alpha}(G)=\alpha D(G)+(1-\alpha)A(G),
$$ 
for any $\alpha\in [0,1]$. The $A_\alpha$-spectral radius of $G$, denoted by $\rho_\alpha(G)$, is the largest eigenvalue of $A_\alpha(G)$, where $\rho_0(G)=\rho(G)$ is the adjacency  spectral radius of $G$ and $2\rho_{\frac{1}{2}}(G)=q(G)$ is the signless Laplacian spectral radius of $G$. Therefore, $A_{\alpha}(G)$ unifies the theories of the adjacency matrix $A(G)$ and the signless Laplacian matrix $Q(G)$. 

By the Perron Frobenius Theorem, the spectral radius of a graph $G$ does not increase when a vertex or an edge  
is removed from a graph $G$. Li et al. \cite{lwm12} and Mieghem et al. \cite{msk11} obtained some results on spectral radius perturbation under vertex and edge deletion. 
Motivated by Hong's inequality $\rho(G)\leq \sqrt{2m-n+1}$ \cite{hong88}, which holds for any simple connected graph $G$ with $n$ vertices and $m$ edges, Guo, Wang, and Li \cite{gwl19} conjectured that for any simple graph $G$ and non-isolated vertex $v$,
$$
\rho^2(G)-\rho^2(G-v)\leq 2d(v)-1.
$$
If this inequality holds, then it would
imply Hong's inequality by simple induction. Sun and Das \cite{sd20} confirmed this conjecture using the eigenvector corresponding to $\rho(G)$, and characterized all connected graphs for which this bound is attained. Utilizing matrix theory, Jin, Zhang and Zhang \cite{jzz24} obtained several upper bounds of the spectral radius of symmetric matrices, which yielded a new proof for the above equality. Recently, Liu and Ning \cite{ln26} provided a short and self-contained proof of this inequality using matrix analysis. 

\begin{theorem}{\rm (Sun and Das \cite{sd20}, Liu and Ning \cite{ln26})}\label{thm:sunliu}
Let $G$ be a simple graph and let $v \in V(G)$ be a non-isolated vertex of degree $d_v$. Then
$$
\rho^2(G)-\rho^2(G-v) \leq 2 d_v-1.
$$
If $G$ is connected, the equality holds if and only if either  $G \cong K_n$, or $G \cong K_{1, n-1}$ and $d_v=1$.
\end{theorem}

In this paper, motivated by Liu and Ning's work \cite{ln26}, 
we establish the corresponding result for $\rho_\alpha(G)$ using matrix analysis, and obtain the main result as follows.

\begin{theorem}\label{thm:alpha spectral two}
Let $G$ be a simple graph of order $n$ and let $v \in V(G)$ be a vertex of degree $d_v$. For every $\alpha\in [0,1)$ and $k\in [0,d_v+1]$, 
\begin{equation}\label{eq:eq1}
\rho^2_{\alpha}(G)- (\rho_{\alpha}(G-v)+\alpha)^2 \leq (1-\alpha)^2 \left(2d_v-1 + \frac{\alpha d_v(n+d_v-1) + (k-2)(d_v-\rho_{\alpha}(G))}{\rho_{\alpha}(G)-\alpha d_v}\right).
\end{equation}
If $G$ is connected, the equality holds if and only if either $G\cong K_n$, or $G\cong K_{1,n-1}$ when $\alpha =0$,  $d_v=1$ and $k=2$.
\end{theorem}

By taking $\alpha=0$ in Theorem \ref{thm:alpha spectral two}, we have the following corollary.

\begin{corollary}\label{cor:alpha=0}
Let $G$ be a simple graph of order $n$ and let $v \in V(G)$ be a vertex of degree $d_v$. For every $k\in [0,d_v+1]$, 
$$
\rho^2(G)- \rho^2(G-v) \leq 2d_v-1 +(k-2)\left(\frac{d_v}{\rho(G)}-1\right),
$$
If $G$ is connected, the equality holds if and only if either $G\cong K_n$, or $G\cong K_{1,n-1}$ when  $d_v=1$ and $k=2$.
\end{corollary}

By setting $k=2$, we obtain Theorem \ref{thm:sunliu} from Corollary \ref{cor:alpha=0} immediately. By setting $\alpha=\frac{1}{2}$ in Theorem \ref{thm:alpha spectral two}, we have the following corollary.

\begin{corollary}\label{cor:alpha=1/2}
Let $G$ be a simple graph of order $n$ and let $v \in V(G)$ be a vertex of degree $d_v$. For every $k\in [0,d_v+1]$,
$$
q^2(G)- (q(G-v)+1)^2 \leq  2d_v-1+\frac{d_v(n+d_v-1)}{q(G)-d_v}+(k-2)\left(\frac{d_v}{q(G)-d_v}-1\right).
$$
If $G$ is connected, the equality holds if and only if $G\cong K_n$.
\end{corollary}

It is therefore natural to investigate the difference between $\rho_{\alpha}(G)$ and $\rho_{\alpha}(G-v)$ for any $v\in V(G)$. Following the proof strategy of Theorem \ref{thm:alpha spectral two}, we present a simple and self-contained proof of the following result, which was first proved by Wang and She. 

\begin{theorem}{\rm (Wang and She \cite{ws22})}\label{thm:alpha spectral}
Let $G$ be a connected graph and let $v \in V(G)$ be a vertex of degree $d_v$. If $\alpha\in [0,1)$, then
\begin{equation}\label{eq:eq2}
\rho_{\alpha}(G)-\rho_{\alpha}(G-v)\leq \alpha +\frac{(1-\alpha)^2d_v}{\rho_{\alpha}(G)-\alpha d_v},
\end{equation}
with equality if and only if $d_v=n-1$ and $G-v$ is regular.
\end{theorem}

By setting $\alpha=0$ and $\alpha=\frac{1}{2}$, we deduce the following corollaries from Theorem \ref{thm:alpha spectral}.

\begin{corollary}{\rm (Wang and She \cite{ws22})}\label{cor:adjacency spectral}
Let $G$ be a connected graph and let $v \in V(G)$ be a vertex of degree $d_v$. Then
$$
\rho(G)-\rho(G-v)\leq \frac{d_v}{\rho(G)},
$$
with equality if and only if $d_v=n-1$ and $G-v$ is regular.
\end{corollary}

By Corollary \ref{cor:adjacency spectral} and $\rho(G)\geq \rho(K_{1,d_v})=\sqrt{d_v}$, it follows that 
$\left(\rho(G)-\frac{d_v}{\rho(G)}\right)^2\leq \rho^2(G-v)$ 
if and only if $\rho(G)-\rho(G-v)\leq \frac{d_v}{\rho(G)}$, 
which yields the following corollary.

\begin{corollary}\label{cor:adjacency spectral two}
Let $G$ be a connected graph and let $v \in V(G)$ be a vertex of degree $d_v$. Then
$$
\rho^2(G)-\rho^2(G-v)\leq 2d_v-\frac{d^2_v}{\rho^2(G)},
$$
with equality if and only if $d_v=n-1$ and $G-v$ is regular.
\end{corollary}

\begin{corollary}{\rm (Wang and She \cite{ws22})}\label{cor:Q spectral}
Let $G$ be a connected graph and let $v \in V(G)$ be a vertex of degree $d_v$. Then
$$
q(G)-q(G-v)\leq \frac{q(G)}{q(G)-d_v},
$$
with equality if and only if $d_v=n-1$ and $G-v$ is regular.
\end{corollary}

If $\rho(G)\geq kd(v)$, then $\rho(G)-\rho(G-v)\leq \frac{1}{k}$ by Corollary \ref{cor:adjacency spectral}. If $q(G)\geq k d(v)$, then $q(G)-q(G-v)\leq \frac{k}{k-1}$ by Corollary \ref{cor:Q spectral}. In particular, $\rho(G)-\rho(G-v)\leq 1$ if $d(v)\leq \rho(G)$, and $q(G)-q(G-v)\leq 2$ if $d(v)\leq \frac{1}{2}q(G)$.

If $d_v \geq \rho(G)$, by substituting $k=0$ into the upper bound in Corollary \ref{cor:alpha=0}, then we have
$$
\rho^2(G)- \rho^2(G-v) \leq 2d_v+1 - \frac{2d_v}{\rho(G)}.
$$
Note that 
$$
2d_v-\frac{d_v^2}{\rho^2(G)} \leq 2d_v+1 - \frac{2d_v}{\rho(G)}\leq 2d_v-1.  
$$
In this sense, the bound in Corollary \ref{cor:adjacency spectral two} is sharper than that in Corollary \ref{cor:alpha=0}, while the bound in Corollary \ref{cor:alpha=0} is sharper than that in Theorem \ref{thm:sunliu}.

If $d_v \leq \rho(G)$, by substituting $k=d_v+1$ into the upper bound in Corollary \ref{cor:alpha=0}, then 
$$
\rho^2(G)- \rho^2(G-v) \leq d_v + \frac{d_v(d_v-1)}{\rho(G)}.
$$
Note that
$$
d_v + \frac{d_v(d_v-1)}{\rho(G)} \leq 2d_v-1 \leq 2d_v-\frac{d_v^2}{\rho^2(G)}.
$$
In this sense, the bound in Corollary \ref{cor:alpha=0} is sharper than that in Theorem \ref{thm:sunliu}, while the bound in Theorem \ref{thm:sunliu} is sharper than that in Corollary \ref{cor:adjacency spectral two}.

In this paper, we also establish the following sharp upper bound of $\rho_\alpha(G)$ in terms of the maximum degree, minimum degree, order and size of $G$. 

\begin{theorem}\label{thm:A_alpha bound}
Let $G$ be a graph of order $n$, with $m$ edges, and with maximum degree $\Delta$ and minimum degree $\delta$. If $\alpha\in [0,1)$, then
\begin{equation}\label{eq:eq3}
\rho_\alpha(G) \leq \min\left\{ \Delta, \frac{1}{2}(\alpha(\Delta+1)+\delta-1) 
+ \frac{1}{2}\sqrt{(\alpha(\Delta+1)-\delta-1)^2+4(1-\alpha)(2m-n\delta)} \right\}.
\end{equation}
Equality holds if and only if $G$ is regular or $G$ has a component of order $\Delta + 1$ in which each vertex
is of degree $\delta$ or $\Delta$, and all other components are $\delta$-regular.
\end{theorem}

A straightforward calculation yields that the right-hand side of \eqref{eq:eq3} equals $\Delta$ provided 
$2m \geq \Delta^2 + \Delta + (n - \Delta - 1) \delta$ for every $\alpha\in [0,1)$.

As an immediate consequence of Theorem \ref{thm:A_alpha bound}, we obtain two known results. First, we retrieve the upper bound of $\rho(G)$, which was independently presented by Hong, Shu, Fang \cite{hsf01} and Nikiforov \cite{niki02}; its equality condition was verified by Zhou and Cho \cite{zc05}. Second, we also arrive at the upper bound of 
$q(G)$ given by Nikiforov \cite{niki14}.

\begin{corollary}{\rm (Hong, Shu and Fang \cite{hsf01})}\label{coro:hsf01}
If $G$ is a graph of order $n$, with $m$ edges, and with minimum degree $\delta$, then 
\[
\rho(G) \leq \frac{1}{2}(\delta-1) + \frac{1}{2}\sqrt{(\delta+1)^2+4(2m - n\delta)}.
\]
Equality holds if and only if $G$ is regular or $G$ has a component of order $\Delta+1$ in which every vertex is of degree $\delta$ or $\Delta$, and all other components are $\delta$-regular.
\end{corollary}

\begin{corollary}{\rm (Nikiforov \cite{niki14})}\label{coro:niki14}
If $G$ is a graph of order $n$, with $m$ edges, with maximum degree $\Delta$ and minimum degree $\delta$, then
\[
q(G) \leq \min\left\{ 2\Delta, \frac{1}{2}\left( \Delta + 2\delta - 1 + \sqrt{(\Delta - 2\delta - 1)^2 + 8(2m - n \delta)} \right) \right\}.
\]
Equality holds if and only if $G$ is regular or $G$ has a component of order $\Delta + 1$ in which every vertex is of degree $\delta$ or $\Delta$, and all other components are $\delta$-regular.
\end{corollary}

The proofs of Theorem \ref{thm:alpha spectral two}, Theorem \ref{thm:alpha spectral} and Theorem \ref{thm:A_alpha bound} will be presented in the following sections.

\section{Proof of Theorem \ref{thm:alpha spectral two}}

For two nonnegative matrices $A=(a_{ij})$ and $B=(b_{ij})$ of the same order, we write $A\leq B$ if $a_{ij}\leq b_{ij}$ for all $i,j$. By the Perron–Frobenius Theorem, if $A\leq B$, then
$\rho(A)\leq \rho(B)$. Moreover, if $A\neq B$ and $B$ is irreducible, then $\rho(A)<\rho(B)$.

\begin{proof}[Proof of Theorem \ref{thm:alpha spectral two}]
We write $A_{\alpha}(G)$ in partitioned form:
$$
A_{\alpha}(G)=\left(\begin{array}{cc}
\alpha d_v & (1-\alpha)\boldsymbol{b}^{\mathrm{T}} \\
(1-\alpha)\boldsymbol{b} & B_{\alpha}
\end{array}\right),
$$
where $\boldsymbol{b}$ is the indicator vector of $N_G(v)$ in $G-v$ and
$B_{\alpha}\leq A_{\alpha}(G-v)+\alpha I$. For simplicity, 
let $\lambda=\rho_{\alpha}(G)$, $\mu=\rho_{\alpha}(G-v)$ and $\theta=\rho(B_\alpha)$. Clearly, $\boldsymbol{b}^{\mathrm{T}} \boldsymbol{b}=d_v$ and $\theta\leq \mu+\alpha$. If $\lambda=\mu$, the result holds immediately. If $\lambda=\theta$, then $G$ is disconnected which implies $\lambda=\mu$. Thus, we may assume that $\lambda>\mu$ and $\lambda>\theta$.

Since $\lambda > \theta$, $\lambda I-B_{\alpha}$ is a positive definite  invertible matrix. By computing the Schur complement of $\lambda I-B_{\alpha}$ in $\lambda I-A_{\alpha}(G)$, we obtain
$$
0=\operatorname{det}(\lambda I-A_{\alpha}(G))=\operatorname{det}(\lambda I-B_{\alpha}) \cdot\left(\lambda-\alpha d_v-(1-\alpha)^2\boldsymbol{b}^{\mathrm{T}}(\lambda I-B_{\alpha})^{-1} \boldsymbol{b}\right),
$$
which implies $\lambda-\alpha d_v=(1-\alpha)^2\boldsymbol{b}^{\mathrm{T}}(\lambda I-B_{\alpha})^{-1} \boldsymbol{b}$ as $\operatorname{det}(\lambda I-B_{\alpha}) \neq 0$. Now, let $t$ be any eigenvalue of $B_{\alpha}$. Since $|t| \leq \theta<\lambda$, we have
$$
f(t):=\frac{\lambda+t}{\lambda^2-\theta^2}-\frac{1}{\lambda-t}=\frac{\theta^2-t^2}{\left(\lambda^2-\theta^2\right)(\lambda-t)}\geq 0.
$$
Since $f(t)$ is an eigenvalue of the matrix $f(B_{\alpha})=\frac{\lambda I+B_{\alpha}}{\lambda^2-\theta^2}-(\lambda I-B_{\alpha})^{-1}$ for any eigenvalue $t$ of $B_{\alpha}$, $f(B_{\alpha})$ is a positive semidefinite matrix. Therefore, we obtain
\begin{equation}\label{eq:eq4}
0\leq \boldsymbol{b}^{\mathrm{T}}f(B_{\alpha})\boldsymbol{b} = 
\frac{\lambda \boldsymbol{b}^{\mathrm{T}}\boldsymbol{b}+\boldsymbol{b}^{\mathrm{T}}B_{\alpha}\boldsymbol{b}}{\lambda^2-\theta^2}-\boldsymbol{b}^{\mathrm{T}}(\lambda I-B_{\alpha})^{-1} \boldsymbol{b} =
\frac{\lambda d_v+\boldsymbol{b}^{\mathrm{T}}B_{\alpha}\boldsymbol{b}}{\lambda^2-\theta^2}-\frac{\lambda-\alpha d_v}{(1-\alpha)^2}.
\end{equation}

Now, we shall estimate $\boldsymbol{b}^{\mathrm{T}}B_{\alpha}\boldsymbol{b}$. 
Let $D=D(G-v)$ and $B=A(G-v)$. Then 
$$
\begin{aligned}
\boldsymbol{b}^{\mathrm{T}}B_{\alpha}\boldsymbol{b}
& =\boldsymbol{b}^{\mathrm{T}}(A_{\alpha}(G-v))\boldsymbol{b}+ \alpha d_v \\
& = \alpha \boldsymbol{b}^{\mathrm{T}} D\boldsymbol{b} +(1-\alpha)\boldsymbol{b}^{\mathrm{T}} B \boldsymbol{b}+ \alpha d_v\\
& = \alpha \left(\sum_{u\in N_G(v)} d_{G-v}(u)+ d_v\right)+(1-\alpha) 2|E(G[N(v)])| \\ 
& = \alpha \left(2|E(G[N(v)])|+ d_G(N(v))\right)+(1-\alpha) 2|E(G[N(v)])| \\
& = \alpha d_G(N(v)) + 2|E(G[N(v)])|\\
& \leq \alpha d_v(n-d_v) + 2|E(G[N(v)])|, 
\end{aligned}
$$
where $d_G(N(v))$ denotes the number of edges between $N(v)$ and $V(G)\setminus N(v)$ in $G$.

To the aim, we first estimate $m:=|E(G[N(v)])|$. 
Let $H=G[N(v) \cup \{v\}]$ be the subgraph of $G$ induced by $N(v) \cup \{v\}$. Then $|V(H)|=d_v+1$ and $|E(H)|=d_v+m$. Since 
\begin{equation}\label{eq:eq5}
\lambda \geq \rho_{\alpha}(H) \geq \frac{2|E(H)|}{|V(H)|}=\frac{2(d_v+m)}{d_v+1},
\end{equation}
then 
\begin{equation}\label{eq:eq6}
2m\leq \lambda(d_v+1)-2d_v. 
\end{equation}
Note that
\begin{equation}\label{eq:eq7}
2m=2|E(G[N(v)])| \leq d_v(d_v-1).
\end{equation}
Let $g(k)= k(d_v-\lambda)+\lambda(d_v+1)-2d_v$. Then $g(k)$ is a linear function of $k$. By \eqref{eq:eq6} and \eqref{eq:eq7}, we have $g(0)\geq 2m$ and $g(d_v+1)=d_v(d_v-1)\geq 2m$,
and so
$$
g(k)=k(d_v-\lambda)+\lambda(d_v+1)-2d_v\geq 2m, 
$$
for each $0\leq k\leq d_v+1$, 
which implies $\boldsymbol{b}^{\mathrm{T}}B_{\alpha}\boldsymbol{b}\leq \alpha d_v(n-d_v)+g(k)$. Substituting this and $\theta\leq \mu+\alpha$ into \eqref{eq:eq4} yields
$$
\begin{aligned}
\lambda^2 & \leq \theta^2+ \frac{(1-\alpha)^2}{\lambda -\alpha d_v}(\lambda d_v +  \boldsymbol{b}^{\mathrm{T}}B_{\alpha}\boldsymbol{b} )\\
& \leq \theta^2+ \frac{(1-\alpha)^2}{\lambda -\alpha d_v}( \lambda d_v+ \alpha d_v(n-d_v) + g(k))\\
&= \theta^2+ (1-\alpha)^2 (2d_v-1) + \frac{(1-\alpha)^2}{\lambda-\alpha d_v}(\alpha d_v(n+d_v-1) + (k-2)(d_v-\lambda)) \\
& \leq  (\mu+\alpha)^2+ (1-\alpha)^2 (2d_v-1) + \frac{(1-\alpha)^2}{\lambda-\alpha d_v}(\alpha d_v(n+d_v-1) + (k-2)(d_v-\lambda)),
\end{aligned}
$$
which is the desired inequality.

Now, we consider the equality case for the connected graph $G$. For the sufficiency, if $G\cong K_n$, then $\rho_{\alpha}(G)=n-1$ and $\rho_{\alpha}(G-v)=n-2$; if $G\cong K_{1,n-1}$, $\alpha =0$, $d_v=1$  and $k=2$, then $\rho_{\alpha}(G)=\sqrt{n-1}$ and $\rho_{\alpha}(G-v)=\sqrt{n-2}$. The equality holds.  

Next, we prove the necessity. If the equality holds in \eqref{eq:eq1}, we have 
$$
\theta=\mu+\alpha \text{ and } \boldsymbol{b}^{\mathrm{T}}f(B_\alpha)\boldsymbol{b}=0. 
$$
Since $f(B_\alpha)$ is a real symmetric and positive semidefinite matrix, there exists a real matrix $C$ such that $f(B_\alpha)=C^{\operatorname{T}}C$. Thus, 
$$
\boldsymbol{b}^{\mathrm{T}}f(B_\alpha)\boldsymbol{b}=\boldsymbol{b}^{\mathrm{T}}C^{\operatorname{T}}C\boldsymbol{b}=0, 
$$
that is $C\boldsymbol{b}=\mathbf{0}$ and so 
$f(B_\alpha)\boldsymbol{b}=C^{\operatorname{T}}C \boldsymbol{b}=\mathbf{0}$. Therefore, 
$$
\mathbf{0}=(\lambda I-B_\alpha)f(B_\alpha)\boldsymbol{b}=(\lambda I-B_\alpha)\left(\frac{\lambda I+B_{\alpha}}{\lambda^2-\theta^2}-(\lambda I-B_{\alpha})^{-1}\right)\boldsymbol{b}, 
$$
which implies $B^2_\alpha\boldsymbol{b}=\theta^2 \boldsymbol{b}=(\mu+\alpha)^2\boldsymbol{b}$.

If $\lambda \leq d_v$, then $g(k)\geq g(0)\geq 2m$. 
If the equality holds in \eqref{eq:eq1},  then $g(k)=g(0)=2m$, and so $k(d_v-\lambda)=0$. It follows that $d_v=\lambda$ or $k=0$.  
If $d_v=\lambda$, then all equalities hold in \eqref{eq:eq6} and \eqref{eq:eq7}, and so $\rho_\alpha(G)=\lambda=d_v=\rho_\alpha(H)$ and $2|E(G[N(v)])|=d_v(d_v-1)$. It follows that $G[N(v)]\cong K_{d_v}$ and so $H\cong K_{d_v+1}$. 
If $k=0$, then the equality holds in \eqref{eq:eq5}, which implies $\rho_\alpha(G)=\lambda=\rho_\alpha(H)=\frac{2|E(H)|}{|V(H)|}$. By the Rayleigh quotient, it follows that $H$ is a regular graph and so $H\cong K_{d_v+1}$. Since $G$ is connected and $\rho_\alpha(G)=\rho_\alpha(H)$, we have $H\cong G$ which means $G\cong K_{n}$.

If $\lambda>d_v$, then $d_v<n-1$ and $g(k)\geq g(d_v+1)\geq 2m$. 
If the equality holds in \eqref{eq:eq1}, then $g(k)=g(d_v+1)=2m$,
and so $k=d_v+1$ and $2m=d_v(d_v-1)$. We may assume that  $\boldsymbol{b}^{\rm T}=(\boldsymbol{1}^{\rm T},\boldsymbol{0}^{\rm T})$ and 
$$
B^2_\alpha=\left(\begin{array}{cc}
B_1 & B_2^{\rm T}\\
B_2 & B_3
\end{array}
\right),
$$
where $B_1$ is the sub-matrix corresponding to $N(v)$. 

By note of $B_\alpha^2\boldsymbol{b}=\theta^2\boldsymbol{b}$, we can get $B_2=O$, which implies that there is no path of length two from $N(v)$ to $V(G)\setminus N[v]$.
Since $G$ is connected, there is at least one edge
$uw$ between $u\in N(v)$ and $w\in V(G)\setminus N[v]$.
Let $B_\alpha=(b_{ij})$. 
Then the $uw$-entry of $B_\alpha^2$ satisfies 
$0=\sum_{j\in V(G-v)}b_{uj}b_{jw}\geq b_{uu}b_{uw}=\alpha(1-\alpha)d_u$. Thus $\alpha=0$.
If $d_v\geq 2$, then by $2m=d_v(d_v-1)$, we can get $G[N(v)]$ is a complete graph, and so there exists a path of length two between $N(v)$ and $V(G)\setminus N[v]$, a contradiction. Thus, $d_v=1$, $k=2$ and $N(v)=\{u\}$. Since
there is no path of length two from $u$ to $V(G)\setminus \{v,u\}$, then $d_u=n-1$ and so $G\cong K_{1,n-1}$. The proof is completed.
\end{proof}

\section{A Short Proof of Theorem \ref{thm:alpha spectral}}

\begin{proof}[Proof of Theorem \ref{thm:alpha spectral}]
We write $A_{\alpha}(G)$ in partitioned form:
$$
A_{\alpha}(G)=\left(\begin{array}{cc}
\alpha d_v & (1-\alpha)\boldsymbol{b}^{\mathrm{T}} \\
(1-\alpha)\boldsymbol{b} & B_{\alpha}
\end{array}\right),
$$
where $\boldsymbol{b}$ is the indicator vector of $N_G(v)$ in $G-v$ and
$B_{\alpha}\leq A_{\alpha}(G-v)+\alpha I$. For simplicity, 
let $\lambda=\rho_{\alpha}(G)$, $\mu=\rho_{\alpha}(G-v)$ and $\theta=\rho(B_\alpha)$. Clearly, $\boldsymbol{b}^{\mathrm{T}} \boldsymbol{b}=d_v$ and $\theta\leq \mu+\alpha$. If $\lambda=\mu$, the result holds immediately. If $\lambda=\theta$, then $G$ is disconnected which implies $\lambda=\mu$. Thus, we may assume that $\lambda>\mu$ and $\lambda>\theta$.

Since $\lambda > \theta$, $\lambda I-B_{\alpha}$ is a positive definite  invertible matrix. By computing the Schur complement of $\lambda I-B_{\alpha}$ in $\lambda I-A_{\alpha}(G)$, we obtain
$$
0=\operatorname{det}(\lambda I-A_{\alpha}(G))=\operatorname{det}(\lambda I-B_{\alpha}) \cdot\left(\lambda-\alpha d_v-(1-\alpha)^2\boldsymbol{b}^{\mathrm{T}}(\lambda I-B_{\alpha})^{-1} \boldsymbol{b}\right),
$$
which implies $\lambda-\alpha d_v=(1-\alpha)^2\boldsymbol{b}^{\mathrm{T}}(\lambda I-B_{\alpha})^{-1} \boldsymbol{b}$ as $\operatorname{det}(\lambda I-B_{\alpha}) \neq 0$. Now, let $t$ be any eigenvalue of $B_{\alpha}$. Since $|t| \leq \theta<\lambda$, we have
$$
f(t):=1-\frac{\lambda-\theta}{\lambda-t}=\frac{\theta-t}{\lambda-t}\geq 0.
$$
Since $f(t)$ is an eigenvalue of the matrix $f(B_{\alpha})=I - (\lambda-\theta)(\lambda I-B_{\alpha})^{-1}$ for any eigenvalue $t$ of $B_{\alpha}$, $f(B_{\alpha})$ is a positive semidefinite matrix. 
Therefore, we obtain
\begin{equation}\label{eq:eq8}
0\leq \boldsymbol{b}^{\mathrm{T}}f(B_{\alpha})\boldsymbol{b} = \boldsymbol{b}^{\mathrm{T}}\boldsymbol{b}-(\lambda-\theta)\boldsymbol{b}^{\mathrm{T}}(\lambda I-B_{\alpha})^{-1} \boldsymbol{b} =
d_v-(\lambda-\theta)\frac{\lambda-\alpha d_v}{(1-\alpha)^2}.
\end{equation}
Thus, by \eqref{eq:eq8} and $\theta\leq \mu+\alpha$, we have
$$
\lambda \leq \theta + \frac{(1-\alpha)^2d_v}{\lambda-\alpha d_v}\leq \mu+\alpha + \frac{(1-\alpha)^2d_v}{\lambda-\alpha d_v},
$$
which is the desired inequality.

Now, we proceed to characterize the case where equality is attained in \eqref{eq:eq2}. For the sufficiency, if $d_v=n-1$ and $G-v$ is regular, then $G-v$ is $\mu$-regular and the quotient matrix of $G$ is 
$$
\left(\begin{array}{cc}
\alpha (n-1) & (1-\alpha)(n-1) \\
(1-\alpha) &  \mu+\alpha 
\end{array}
\right), 
$$
which yields $(\lambda-\alpha(n-1))(\lambda-\mu-\alpha)=(1-\alpha)^2(n-1)$ and the equality is attained.  

Next, we prove the necessity. If the equality holds in \eqref{eq:eq2}, we have 
$$
\theta=\mu+\alpha \text{ and } \boldsymbol{b}^{\mathrm{T}}f(B_\alpha)\boldsymbol{b}=0.
$$
Since $f(B_\alpha)$ is a real symmetric and positive semidefinite matrix, 
$\boldsymbol{b}^{\mathrm{T}}f(B_\alpha)\boldsymbol{b}=0$ implies
$f(B_\alpha)\boldsymbol{b}=\mathbf{0}$. Therefore, 
$$
\mathbf{0}=(\lambda I-B_\alpha)f(B_\alpha)\boldsymbol{b}=(\lambda I-B_\alpha)\boldsymbol{b} - (\lambda-\theta)\boldsymbol{b}, 
$$
which implies $B_\alpha\boldsymbol{b}=\theta \boldsymbol{b}$.

Let $G_1,\dots,G_s$ be the components of $G-v$ and $B_\alpha={\rm diag}(B_1,\dots,B_s)$ with a suitable relabeling of the vertices of $G-v$. Suppose that $\boldsymbol{b}^{\operatorname{T}}=(\boldsymbol{b}_1^{\operatorname{T}},\dots,\boldsymbol{b}_s^{\operatorname{T}})$ such that $B_i\boldsymbol{b}_i=\theta\boldsymbol{b}_i$ for each $i\in \{1,\dots,s\}$. Since $G$ is connected and $v$ has at least one neighbor in each $G_i$, we have $\boldsymbol{b}_i\neq \mathbf{0}$ for each $i\in \{1,\dots,s\}$. Therefore, $\boldsymbol{b}_i$ is an eigenvector of $B_i$ corresponding to the eigenvalue $\theta$ for each $i\in \{1,\dots,s\}$. Since $\theta$ is the largest eigenvalue of $B_\alpha$, we have $\theta$ is also the largest eigenvalue of each $B_i$ and so $\boldsymbol{b}_i$ is a positive eigenvector. Since $\boldsymbol{b}_i$ is an indicator vector, we conclude that $\boldsymbol{b}_i=\boldsymbol{1}$. Therefore, $d_v=n-1$ and 
$B_{\alpha} = A_{\alpha}(G-v)+\alpha I$. Thus, 
$$
(A_{\alpha}(G-v)+\alpha I)\boldsymbol{1}=B_\alpha\boldsymbol{1}=\theta \boldsymbol{1}=(\mu+\alpha) \boldsymbol{1}, 
$$
which yields $A_{\alpha}(G-v)\boldsymbol{1}=\mu \boldsymbol{1}$. Since each row sum of $A_{\alpha}(G-v)$ is the degree of corresponding vertex,   $G-v$ is $\mu$-regular.
\end{proof}

\section{Proof of Theorem \ref{thm:A_alpha bound}}

To the aim, we first clarify the dependence of the right-hand side of \eqref{eq:eq3} on $\delta$ and $\Delta$ via the Lemma  below.

\begin{lemma}\label{lem:f(a,b)}
Let $n \geq 1$ and $0\leq a <\frac{2m}{n} < b\leq n-1$. Then the function
$$
f(a,b)=  \frac{1}{2}(\alpha(b+1)+a-1)
+ \frac{1}{2}\sqrt{(\alpha(b+1)-a-1)^2+4(1-\alpha)(2m-na)}
$$
is decreasing in $a$ for $\alpha\in [0,1)$ and increasing in $b$ for $\alpha\in (0,1]$.
\end{lemma}

\begin{proof}
Let $M=\sqrt{(\alpha(b+1)-a-1)^2+4(1-\alpha)(2m-na)}$. On one hand, 
$$
\frac{\partial f(a,b)}{\partial a} = \frac{M-[\alpha(b+1)-a-1+2(1-\alpha)n]}{2M}.
$$
Since $\alpha(b+1)-a-1+2(1-\alpha)n>(1-\alpha)(2n-a-1)\geq 0$ and  
$$
\begin{aligned}
 &M^2-[\alpha(b+1)-a-1+2(1-\alpha)n]^2 \\
 =&4(1-\alpha)[2m-na-(\alpha(b+1)-a-1)n-(1-\alpha)n^2]\\
 =&4(1-\alpha)[2m-(1-\alpha)n(n-1)-\alpha n b]\\
 < &4(1-\alpha)[2m-(1-\alpha)n(n-1)-2m \alpha] \\
 = &4(1-\alpha)^2[2m-n(n-1)]< 0, 
\end{aligned}
$$
we have $M<\alpha(b+1)-a-1+2(1-\alpha)n$. 
Thus, $\frac{\partial f}{\partial a}<0$ and so $f(a,b)$ is decreasing in $a$ for $\alpha\in [0,1)$. 

On the other hand, 
$$
\frac{\partial f(a,b)}{\partial b} = \alpha \cdot \frac{M+[\alpha(b+1)-a-1]}{2M}.
$$
If $\alpha(b+1)\geq a+1$, then $\frac{\partial f}{\partial b}>0$ for $\alpha\in (0,1)$ as $M>0$. If $\alpha(b+1)< a+1$, then $M > a+1-\alpha(b+1)$ for $\alpha\in (0,1)$ and so $\frac{\partial f}{\partial b}>0$. If $\alpha =1$, then $f(a,b)=b$. Hence, $f(a,b)$ is increasing in $b$ for $\alpha\in (0,1]$.
\end{proof}

\begin{proof}[Proof of Theorem \ref{thm:A_alpha bound}]
Let $r_i(C)$ denote the $i$-th row-sum of a matrix $C$ and let $d_v$ be the degree of vertex $v$ in $G$. Denote $A_\alpha = A_\alpha(G)= \alpha D(G)+(1-\alpha)A(G)$ and $\rho_\alpha(G)=  \rho(A_\alpha(G))$. Define the matrix
$M_\alpha$ as 
$$
M_\alpha=A_\alpha^2-(\alpha(\Delta+1)+\delta-1)A_\alpha.
$$
Then, 
$$
r_v\left(A_\alpha \right)= \alpha d_v +(1-\alpha) d_v = d_v.
$$
We will show that the row-sums of $M_\alpha$ do not exceed $2(1-\alpha)m-((1-\alpha)(n-1) +\alpha \Delta)\delta$. Recall that 
$$
r_v\left(A_\alpha^2\right)= \alpha d_v^2+(1-\alpha) \sum_{u \in N(v)} d_u.
$$
Now, it is clear that 
$$
\sum_{u \in N(v)} d_u=\sum_{u \in V(G)} d_u-\sum_{ u \in V(G)\setminus N(v)} d_u \leq 2 m-d_v-\left(n-d_v-1\right) \delta,
$$
and so
$$
\begin{aligned}
r_v(M_\alpha) & =\alpha d_v^2+ (1-\alpha) \sum_{u \in N(v)} d_u-(\alpha(\Delta+1)+\delta-1) d_v \\
& \leq (1-\alpha)(2m- d_v-(n-d_v-1) \delta )+\alpha d_v^2-(\alpha(\Delta+1)+\delta-1) d_v \\
& =(1-\alpha)(2m-(n-1) \delta)+ \alpha d_v^2+ (1-\alpha)(\delta-1)d_v -(\alpha(\Delta+1)+\delta-1) d_v \\
& =(1-\alpha)(2m-(n-1) \delta)+ \alpha d_v^2-\alpha(\Delta+\delta) d_v.
\end{aligned}
$$
For the purpose of eliminating the term $\alpha d_v^2$, we shall prove that for every vertex $v \in V(G)$,
$$
\alpha d_v^2-\alpha(\Delta+\delta) d_v \leq -\alpha \Delta\delta.
$$
In fact, the function $f(t)=t^2-(\Delta+\delta) t$ is convex and, in light of $\Delta \geq d_v \geq \delta$, we conclude that
$$
\begin{aligned}
d_v^2- (\Delta+\delta) d_v & =f(d_v) \leq \max \{f(\Delta), f(\delta)\} \\
& = \max \left\{\Delta^2-(\Delta+\delta) \Delta, \delta^2-(\Delta+\delta) \delta\right\} \\
& =-\Delta\delta.
\end{aligned}
$$
Hence, by $\alpha\geq 0$, 
$$
\alpha d_v^2 \leq \alpha(\Delta+\delta) d_v-\alpha \Delta\delta.
$$
Substituting this into $r_v(M_\alpha)$, we have 
$$
\begin{aligned}
r_v(M_\alpha) & \leq 2(1-\alpha)m-(1-\alpha)(n-1) \delta+ \alpha d_v^2-\alpha(\Delta+\delta) d_v \\
& \leq 2(1-\alpha)m-((1-\alpha)(n-1) +\alpha \Delta)\delta.
\end{aligned}
$$
Denote by $\rho(M_\alpha)$ the largest eigenvalue of $M_\alpha$. Then we obtain
$$
\rho(M_\alpha) \leq \max_{v \in V(G)} r_v(M_\alpha) \leq 2(1-\alpha)m-((1-\alpha)(n-1) +\alpha \Delta)\delta.
$$
On the other hand, writing $\rho_\alpha$ for $\rho_\alpha(G)$, it follows that $\rho_\alpha^2-(\alpha(\Delta+1)+\delta-1)\rho_\alpha$ is an eigenvalue of $M_\alpha$ and so,
$$
\rho_\alpha^2-(\alpha(\Delta+1)+\delta-1)\rho_\alpha-2(1-\alpha)m+((1-\alpha)(n-1) +\alpha \Delta)\delta \leq 0 .
$$
Solving this inequality, and by $\rho_{\alpha}(G)\leq \Delta$, we establish \eqref{eq:eq3}.

We now characterize the equality case in \eqref{eq:eq3}. First, assume that $G$ is a  connected graph. If the equality holds in \eqref{eq:eq3}, then all row-sums of $M_\alpha$ are equal and so
$$
\sum_{u \in V(G)\setminus N[v]} d_u= (n-d_v-1) \delta
$$
for each $v \in V(G)$. Therefore, if $uv\notin E(G)$, then $d_u\neq n-1$ and $d_v\neq n-1$, and so $d_v=\delta$ and $d_u=\delta$, which implies that $G$ is regular or each vertex is of degree either $\delta$ or $n-1$.

Now suppose that $G$ is not connected and the equality holds in \eqref{eq:eq3}. Then $G$ contains a component $G_1$ such that $\rho_\alpha(G_1)=\rho_\alpha(G)$. Set $n_1=|V(G_1)|$, $m_1=|E(G_1)|$, $\Delta_1=\Delta(G_1)$, and $\delta_1=\delta(G_1)$. In view of $\delta \leq \delta_1$ and $\Delta \geq \Delta_1$, inequality \eqref{eq:eq3} and Lemma \ref{lem:f(a,b)} imply that
$$
\begin{aligned}
\rho_\alpha(G_1)  \leq & \frac{1}{2}(\alpha(\Delta_1+1)+\delta_1-1) 
 + \frac{1}{2}\sqrt{(\alpha(\Delta_1+1)-\delta_1-1)^2+4(1-\alpha)(2m_1-n_1\delta_1)} \\
\leq & \frac{1}{2}(\alpha(\Delta+1)+\delta-1) 
 + \frac{1}{2}\sqrt{(\alpha(\Delta+1)-\delta-1)^2+4(1-\alpha)(2m_1-n_1\delta)} \\
\leq & \frac{1}{2}(\alpha(\Delta+1)+\delta-1) 
 + \frac{1}{2}\sqrt{(\alpha(\Delta+1)-\delta-1)^2+4(1-\alpha)(2m-n\delta)}\\
 = & \rho_\alpha(G), 
\end{aligned}
$$
where the last inequality holds as $2(m-m_1)\geq (n-n_1)\delta$. Since $\rho_\alpha(G_1)=\rho_\alpha(G)$, it follows that all inequalities of the above chain are equalities. In particular, all vertices belong to $G-G_1$ have degree precisely $\delta$. Therefore, for any given $v \in V(G_1)$ and for each $u \in V(G_1)$, 
$$
\sum_{u \in V(G_1)\setminus N[v]} d_{G_1}(u)= (n_1-d_{G_1}(v)-1) \delta. 
$$
Thus, either $G_1$ is $\delta$-regular or $n_1=\Delta+1$ and each vertex of $G_1$ is of degree $\delta$ or $\Delta$. This completes the proof of the necessity of the condition for equality in \eqref{eq:eq3}. 

Now, we prove the sufficiency of the condition. We only consider the connected graph. 
Let $G$ be a connected graph of order $n$ in which each vertex is of degree $\delta$ and $n-1$. The quotient matrix of $A_\alpha(G)$ is
$$
P=\left(\begin{array}{cc}
\alpha n-1+k(1-\alpha)  & (1-\alpha)(n-k) \\
(1-\alpha)k & \delta-(1-\alpha)k
\end{array}
\right),
$$
where $k$ is the number of the vertices of degree $n-1$. 
By $2m-n\delta=k(n-1-\delta)$, we have 
$$
\begin{aligned}
{\rm det} (\lambda I -P) & =\lambda^2-(\alpha n+\delta-1)\lambda +[\alpha n-1+k(1-\alpha)][\delta-(1-\alpha)k]-(1-\alpha)^2 k(n-k) \\
& =\lambda^2-(\alpha n+\delta-1)\lambda + (\alpha n-1)\delta -(1-\alpha) k (n-\delta-1)\\
& = \lambda^2-(\alpha n+\delta-1)\lambda + (\alpha n-1)\delta - (1-\alpha)(2m-n\delta).
\end{aligned}
$$
Thus, 
$$
\rho_\alpha(G) = \frac{1}{2}(\alpha n+\delta-1)+\frac{1}{2}\sqrt{(\alpha n-\delta-1)^2+(1-\alpha)(2m-n\delta)}, 
$$
which satisfies the equality in \eqref{eq:eq3}. 
The proof is completed.
\end{proof}

Huang, Lin and Xue \cite{hlx20} obtained an upper bound on $\rho_{\alpha}(G)$ for $\alpha \in [0,1)$, which generalizes the upper bound on $\rho_0(G)$ independently obtained by Hong, Shu, Fang \cite{hsf01} and Nikiforov \cite{niki02}. 

\begin{theorem}{\rm (Huang, Lin and Xue \cite{hlx20})}\label{thm:HLX20}
Let $G$ be a graph of order $n$, with $m$ edges, and with  maximum degree $\Delta$ and minimum degree $\delta$. If $\alpha \in [0,1)$, then
\[
\rho_{\alpha}(G) \leq \frac{1}{2}(1-\alpha)(\delta-1) + \frac{1}{2}\sqrt{(1-\alpha)^2(\delta-1)^2 + 4\big(\alpha\Delta^2 + (1-\alpha)(2m - (n-1)\delta)\big)}.
\]
The equality holds if and only if one of the following statements hold.
\begin{enumerate}[$(i)$]
    \item For $\alpha = 0$, $G$ is a disjoint union of a $\delta$-regular graph and a bidegree graph of order $\Delta+1$ and degree either $\delta$ or $\Delta$.
    \item For $0 < \alpha < 1$, $G$ is a $\Delta$-regular graph if $G$ is connected, and $G$ is the union of a complete graph with order $\Delta+1$ and a $\delta$-regular graph if $G$ is disconnected.
\end{enumerate}
\end{theorem}

For $0 < \alpha < 1$, the upper bound in Theorem \ref{thm:A_alpha bound} is sharper than that in Theorem \ref{thm:HLX20}. In fact, consider
$$
\begin{aligned}
f(x)&=x^2-(\alpha(\Delta+1)+\delta-1)x-2(1-\alpha)m+((1-\alpha)(n-1) +\alpha \Delta)\delta, \\
g(x)&=x^2-(1-\alpha)(\delta-1)x-\alpha\Delta^2-(1-\alpha)(2m-(n-1)\delta). 
\end{aligned}
$$
Then $g(x)-f(x)=\alpha(\Delta+\delta)(x-\Delta)\leq 0$ if $x\leq \Delta$. Since $\rho_\alpha\leq \Delta$, we have $g(\rho_\alpha)\leq f(\rho_\alpha)$.

\section*{Declaration of competing interest}

There is no competing interest.

\section*{Data availability}

No data was used for the research described in the article.

\section*{Acknowledgement}
The research of Zhen-Mu Hong is supported by Natural Science Foundation of China (No. 12371338) and Outstanding Youth Scientific Research Projects of Anhui Provincial Department of Education (No. 2022AH030073). The research of Zheng-Jiang Xia is supported by
Key Project in Natural Science Research of Anhui Provincial Department of Education (No. 2023AH050268). 
The research of Zhi Qiao is supported by Natural Science Foundation of China (No. 12071321).

\end{document}